\documentclass[11pt]{article}
\usepackage[cp1251]{inputenc}
\usepackage[english]{babel}
\usepackage{amsmath}
\usepackage{amsfonts}
\usepackage{amssymb}
\usepackage{graphicx}
\newtheorem{theorem}{\textsf{Theorem}}
\newtheorem{lemma}{\textsf{Lemma}}

\newenvironment{proof}[1][\textsf{Proof. }]{\textbf{#1}}{$\square$}
\def\text{\hbox} 
\def\tr{{\rm tr\,}}

\begin{document}

\title{
\date{}
{
\large \textsf{\textbf{Examples of rigid and flexible Seifert fibred cone-manifolds}}
}
}
\author{\small Alexander Kolpakov\footnote{Supported by the Schweizerischer Nationalfonds SNF no.~200020-121506/1 and no.~200021-131967/1}}
\maketitle

\begin{abstract}\noindent
The present paper gives an example of a rigid spherical cone-manifold and that of a flexible one which are both Seifert fibred.


\end{abstract}

\parindent=0pt
\section{Introduction}

The theory of three-dimensional orbifolds and cone-manifolds  attracts attention of many mathematicians since the original work of Thurston \cite{Thurston}. An introduction to the theory of orbifolds could be found in \cite[Ch.13]{Thurston}. For a basic introduction to the geometry of three-dimensional cone-manifolds and cone-surfaces we refer the reader to \cite{CooperHodgsonKerckhoff}. The main motivation for studying three-dimensional cone-manifolds comes from Thurston's approach to geometrization of three-orbifolds: three-dimensional cone-manifolds provide a way to deform geometric orbifold structures. The orbifold theorem has been proven in full generality by M.~Boileau, B.~Leeb and J.~Porti, see \cite{BoileauLeebPorti2001, BoileauLeebPorti2005}.

One of the main questions in the theory of three-dimensional cone-manifolds is the rigidity problem. First, the rigidity property was discovered for hyperbolic manifolds (so-called Mostow-Prasad rigidity, see \cite{Mostow, Prasad}). After that, the global rigidity property for hyperbolic three-dimensional cone-manifolds with singular locus a link and cone angles less than $\pi$ was proven by S. Kojima \cite{Kojima}. The key result that implies global rigidity is due to Hodgson and Kerckhoff \cite{HK98}, who showed the local rigidity of hyperbolic cone manifolds with singularity of link or knot type and cone angles less than $2\pi$. The de~Rham rigidity for spherical orbifolds was established in \cite{deRham, Rothenberg}. Detailed analysis of the rigidity property for three-dimensional cone-manifolds was carried out in \cite{Weiss2005, Weiss2007} for hyperbolic and spherical cone-manifolds with singularity a trivalent graph and cone angles less than $\pi$. 

Recently, the local rigidity for hyperbolic cone-manifolds with cone angles less than $2\pi$ was proven in \cite{M2009, Weiss2009}. However, examples of infinitesimally flexible hyperbolic cone-manifolds had already been given in \cite{Casson}. For other examples of flexible cone-manifolds one may refer to \cite{Izmestiev2009, Porti2002, Schlenker}.

The theorem of \cite{Weiss2007} concerning the global rigidity for spherical three-dimensional cone-manifolds was proven under the condition of being \emph{not Seifert fibred}. Recall that due to \cite{Porti2004} a cone-manifold is \emph{Seifert fibred} if its underlying space carries a Seifert fibration such that components of the singular stratum are leafs of the fibration. In particular, if its singular stratum is represented by a link, then the complement is a Seifert fibred three-manifold. All Seifert fibred link complements in the three-sphere are described by \cite{BurdeMurasugi}. In the present paper, we give an explicit example of a rigid spherical cone-manifold and a flexible one which are both Seifert fibred. The singular locus for each of these cone-manifolds is a link and the underlying space is the three-sphere $\mathbb{S}^3$. The rigid cone-manifold given in the paper has cone-angles of both kinds, less or greater than $\pi$. The flexible one has cone-angles strictly greater than $\pi$. Deformation of its geometric structure comes essentially from those of the base cone-surface. However, hyperbolic orbifolds, which are Seifert fibred over a disc, are rigid. Their geometric structure degenerates to the minimal-perimeter hyperbolic polygon, as shown in \cite{Porti2010}. These are uniquely determined by cone angles.

The paper is organised as follows: first, we recall some common facts concerning spherical geometry. In the second section, the geometry of the Hopf fibration is considered and a number of lemmas are proven. After that, we construct two explicit examples of Seifert fibred cone-manifolds. The first one is a globally rigid cone-manifold and its moduli space is parametrised by its cone angles only. The second one is a flexible Seifert fibred cone-manifold. This means that we can deform its metric while keeping its cone angles fixed. Rigorously speaking, the following assertion is proven: the given cone-manifold has a one-parameter family of distinct spherical cone metrics with the same cone angles.

\bigskip
{\textbf{Acknowledgement}. The author is grateful to Prof. J. Porti (Universitat Aut\`{o}noma de Barcelona) and Prof. J.-M. Schlenker (Institut de Math\'{e}matiques de Toulouse) for their valuable comments on the paper and discussion of the subject.}
\section{Spherical geometry}

Below we present several common facts concerning spherical geometry in dimension two and three.

Let us identify a point $p=(w,x,y,z)$ of the three-dimensional sphere
\begin{equation*}
\mathbb{S}^3 = \{(w, x, y, z) \in \mathbb{R}^4 | w^2+x^2+y^2+z^2 = 1 \}
\end{equation*}
with an $SU_2(\mathbb{C})$ matrix of the form
\begin{equation*}
P = \left(
      \begin{array}{cc}
        w + i x & y + i z \\
        -y + i z & w - i x \\
      \end{array}
    \right).
\end{equation*}

Then, replace the group ${\rm Isom}^+\, \mathbb{S}^3 \cong SO_4(\mathbb{R})$ of orientation preserving isometries with its two-fold covering $SU_2(\mathbb{C})\times SU_2(\mathbb{C})$. Finally, define the action of $\langle A, B \rangle \in SU_2(\mathbb{C})\times SU_2(\mathbb{C})$ on $P \in SU_2(\mathbb{C})$ by
\begin{equation*}
\langle A, B \rangle : P \longmapsto A^t P \overline{B}.
\end{equation*}
Thus, we define the action of $SO_4(\mathbb{R}) \cong SU_2(\mathbb{C})\times SU_2(\mathbb{C})\slash\{\pm {\rm id}\}$ on the three-sphere~$\mathbb{S}^3$.

By assuming $w=0$, we obtain the two-dimensional sphere
$$\mathbb{S}^2 = \{ (x, y, z) \in \mathbb{R}^3 | x^2+y^2+z^2 = 1 \}.$$
Let us identify a point $(x, y, z)$ of $\mathbb{S}^2$ with the matrix
\begin{equation*}
Q = \left(
      \begin{array}{cc}
        i x & y + iz \\
        -y + iz & - i x \\
      \end{array}
    \right),
\end{equation*}
which represents a pure imaginary unit quaternion $Q \in \mathbf{H}$.

Instead of ${\rm Isom}^+\,\mathbb{S}^2 \cong SO_3(\mathbb{R})$ we use its two-fold covering $SU_2(\mathbb{C})$ acting by
\begin{equation*}
A : q \longmapsto A^t q \overline{A}
\end{equation*}
for every $A \in SU_2(\mathbb{C})$ and every $q \in \mathbb{S}^2$.

Equip each $\mathbb{S}^3$ and $\mathbb{S}^2$ with an intrinsic metric of constant sectional curvature $+1$. We call the distance between two points $P$ and $Q$ of $\mathbb{S}^n$ ($n = 2,3$) a real number $d(P,Q)$ uniquely defined by the conditions
\begin{equation*}
0 \leq d(P,Q) \leq \pi,
\end{equation*}
\begin{equation*}
\cos d(P,Q) = \frac{1}{2}\,{\rm tr}\,P^t\overline{Q}.
\end{equation*}

The next step is to describe spherical geodesic lines in $\mathbb{S}^n$. Let us recall the following theorem \cite[Theorem 2.1.5]{Ratcliffe}.

\begin{theorem}
A function $\lambda : \mathbb{R} \rightarrow \mathbb{S}^n$ is a geodesic line if and only if there are orthogonal vectors $x$, $y$ in $\mathbb{S}^n$ such that
\begin{equation*}
\lambda(t) = (\cos t) x + (\sin t) y.
\end{equation*}
\end{theorem}

Taking into account the preceding discussion, we may reformulate the statement above.

\begin{lemma}
Every geodesic line (a great circle) in $\mathbb{S}^3$ (respectively, $\mathbb{S}^2$) could be represented in the form
\begin{equation*}
C(t) = P \cos t + Q \sin t,
\end{equation*}
where $P, Q \in SU_2(\mathbb{C})$ (respectively $P, Q \in \mathbf{H}$) satisfy orthogonality condition
\begin{equation*}
\cos\,d(P,Q) = 0.
\end{equation*}
\end{lemma}

By virtue of this lemma, one may regard $P$ as \textit{the starting point} of the curve $C(t)$ and $Q$ as \textit{the velocity vector} at $P$, since $C(0) = P$, $\dot C(0) = \frac{d}{dt}\,C(t)|_{t=0} = Q$ and $d(C(0),\dot{C}(0)) = \frac{\pi}{2}$ (the latter holds up to a change of the parameter sign).

Given two geodesic lines $C_1(t)$ and $C_2(t)$, define their common perpendicular $C_{12}(t)$ as a geodesic line such that there exist $0 \leq t_1, t_2 \leq 2\pi$, $0 \leq \delta \leq \pi$ with the following properties:
\begin{equation*}
C_{12}(0) = C_1(t_1),\,C_{12}(\delta) = C_2(t_2),
\end{equation*}
\begin{equation*}
d(\dot{C}_{12}(0), \dot{C}_1(t_1)) = d(\dot{C}_{12}(\delta), \dot{C}_2(t_2))=\frac{\pi}{2}.
\end{equation*}

We call $\delta$ the distance between the geodesics $C_1(t)$ and $C_2(t)$. Note, that for an arbitrary pair of geodesics their common perpendicular should not be unique.

For an additional explanation of spherical geometry we refer the reader to \cite{Ratcliffe} and \cite[Chapter~6.4.2]{Weiss2005}.
\section{Links arising from the Hopf fibration}

The present section is devoted to the construction of a family of links $\mathcal{H}_n$ ($n\geq2$) which we shall use later. These links have a nice property -- each of them is formed by $n\geq2$ fibres of the Hopf fibration. Recall that the Hopf map $h : \mathbb{S}^3 \xrightarrow{\mathbb{S}^1} \mathbb{S}^2$ has geometric nature \cite[p.~654]{Hopf}. Our aim is to prove a number of lemmas concerning the geometry of the Hopf fibration in more detail.
\subsection{Links $\mathcal{H}_n$ as fibres of the Hopf fibration}

The Hopf map $h$ is defined as follows \cite{Hopf}: for every point $(w, x, y, z) \in \mathbb{S}^3$ let its image on $\mathbb{S}^2$ be
\begin{equation*}
h(w,x,y,z) = \left(2(xz+wy), 2(yz-wx), 1-2(x^2+y^2)\right).
\end{equation*}

The fibre $h^{-1}(a,b,c)$ over the point $(a,b,c) \in \mathbb{S}^2$ is a geodesic line in $\mathbb{S}^3$ of the form
\begin{equation*}
C(t) = \frac{1}{\sqrt{2(1+c)}}\left(\left(1+c, - b, a, 0\right)\,\cos t+
(0, a, b, 1+c)\,\sin t\right).
\end{equation*}

The exceptional point $(0,0,-1)$ has the fibre $(0,\cos t,- \sin t,0)$.

The line $C(t)$ is a great circle of $\mathbb{S}^3$ and can be rewritten in the matrix form
\begin{equation*}
C(t) = P(a,b,c) \cos t + Q(a,b,c) \sin t,
\end{equation*}
where
\begin{equation*}
P(a,b,c) = \frac{1}{\sqrt{2(1+c)}}\,\left(\begin{array}{cc}(1+c) - ib & a \\-a & (1+c) + ib \\\end{array}\right),
\end{equation*}
\begin{equation*}
Q(a,b,c) = P(a,b,c)\,\left(\begin{array}{cc}0 & i \\i & 0 \\\end{array}\right).
\end{equation*}

We call
\begin{equation*}
F(t) = \left(\begin{array}{cc}1 & 0 \\0 & 1 \\\end{array}\right) \cos t + \left(\begin{array}{cc}0 & i \\i & 0 \\\end{array}\right) \sin t
\end{equation*}
\textit{the generic fibre} $h^{-1}(0, 0, 1)$. Moreover, every fibre $h^{-1}(a, b, c)$ can be described as a circle $C(t) = P(a,b,c)\,F(t)$. Note, that
$P(a,b,c)$ is an $SU_2(\mathbb{C})$ matrix. Thus $C(t)$ could be obtained from $F(t)$ by means of the isometry $\langle P(a,b,c)^t, {\rm id}\rangle$. For the exceptional point $(0, 0, -1) \in \mathbb{S}^2$, we set
\begin{equation*}
P(0,0,-1) = \left(\begin{array}{cc}0 & 1 \\-1 & 0 \\\end{array}\right).
\end{equation*}

It is known, that every pair of distinct fibres of the Hopf fibration represents simply linked circles in $\mathbb{S}^3$ (the Hopf link). Thus, $n$ fibres form a link $\mathcal{H}_n$ whose every two components form the Hopf link. One can obtain it by drawing $n$ straight vertical lines on a cylinder and identifying its ends by a rotation through the angle of $2\pi$. Hence $\mathcal{H}_n$ is an $(n, n)$ torus link.

\begin{figure}[ht]
\begin{center}
\includegraphics* [totalheight=4cm]{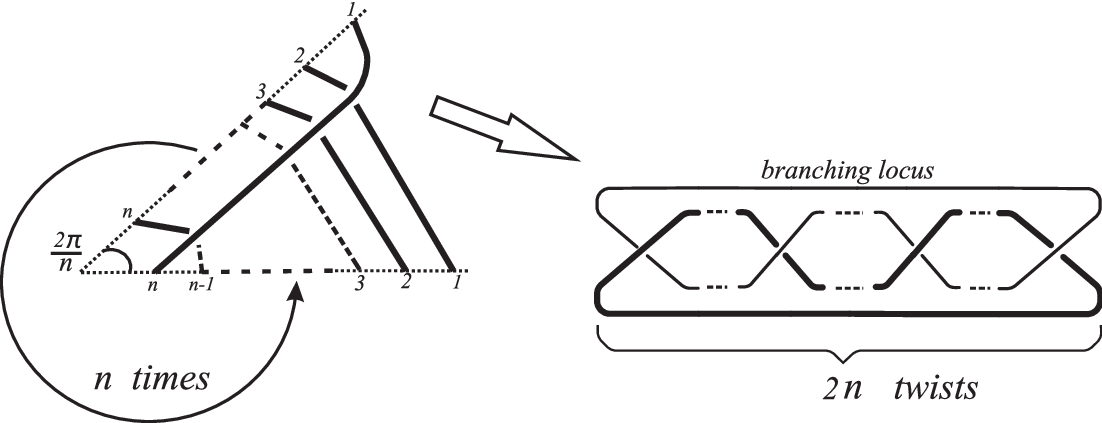}
\end{center}
\caption{$n$-fold branched covering of $(2,2n)$ torus link by $\mathcal{H}_n$} \label{nfold}
\end{figure}

Another remark is that the $\mathcal{H}_n$ link could be arranged around a point in order to reveal its $n$-th order symmetry, as depicted in Fig.~\ref{nfold}. This fact allows us to consider $n$-fold branched coverings of the corresponding cone-manifolds with singular locus $\mathcal{H}_n$ that appear in Section~$4$.
\subsection{Geometry of the Hopf fibration}

Here and below we use the polar coordinate system $(\psi, \theta)$ on $\mathbb{S}^2$ instead of the Cartesian one. Suppose
\begin{equation*}
a = \cos \psi \sin \theta,\,\,b = \sin \psi \sin \theta,\,\,c = \cos \theta,
\end{equation*}
\begin{equation*}
0 \leq \psi \leq 2\pi,\,\, 0 \leq \theta \leq \pi
\end{equation*}
and let
\begin{equation*}
M(\psi,\theta) = P(a, b, c) = \left(\begin{array}{cc}
                                  \cos\frac{\theta}{2} - i \sin\psi \sin\frac{\theta}{2} & \cos\psi \sin\frac{\theta}{2} \\
                                - \cos\psi \sin\frac{\theta}{2} & \cos\frac{\theta}{2} + i \sin\psi \sin\frac{\theta}{2} \\
                                \end{array}\right).
\end{equation*}

A rotation of $\mathbb{S}^3$ about the generic fibre $F(t)$ through angle $\omega$ has the form $\langle R(\omega), R(\omega) \rangle$, where
\begin{equation*}
R(\omega) = \left(\begin{array}{cc}\cos\frac{\omega}{2} & i \sin\frac{\omega}{2} \\i \sin\frac{\omega}{2} & \cos\frac{\omega}{2} \\\end{array}\right).
\end{equation*}

The image of $F(t)$ under the Hopf map $h$ is $(0, 0)$ w.r.t. the polar coordinates. The following lemma shows how to obtain a rotation about the pre-image $h^{-1}(\psi, \theta)$ of an arbitrary point $(\psi, \theta)$.

\begin{lemma}\label{lemma:rotation}
A rotation through angle $\omega$ about an axis $C(t)$ in $\mathbb{S}^3$ which is the pre-image of a point $(\psi, \theta) \in \mathbb{S}^2$ with respect to the Hopf map is
\begin{equation*}
\langle \overline{M(\psi,\theta)}R(\omega)M(\psi,\theta)^t, R(\omega) \rangle.
\end{equation*}
\end{lemma}
\begin{proof}
Since we have that $C(t) = M(\psi, \theta) F(t)$ and $R(\omega)^tF(t)\overline{R(\omega)} = F(t)$ for every $0\leq t \leq 2\pi$, then
\begin{equation*}
\left(\overline{M(\psi,\theta)}R(\omega)M(\psi,\theta)^t\right)^tC(t)\overline{R(\omega)} = M(\psi, \theta) R(\omega)^tF(t)\overline{R(\omega)} =
\end{equation*}
\begin{equation*}
= M(\psi, \theta) F(t) = C(t)
\end{equation*}
by a straightforward computation.
Here we use the fact that $M(\psi, \theta) \in SU_2(\mathbb{C})$, and so $\overline{M(\psi,\theta)^t}M(\psi,\theta) = {\rm id}$.
\end{proof}

Another remarkable property of the Hopf fibration is discussed below.

\begin{lemma}\label{lemma:equidistant}
Every two fibres $C_1(t)$ and $C_2(t)$ of the Hopf fibration are equidistant geodesic lines (great circles) in $\mathbb{S}^3$.

If $C_i(t)$, $i\in\{1,2\}$ are pre-images of the points $\widehat{C}_i\in\mathbb{S}^2$, then the length $\delta$ of the common perpendicular for $C_1(t)$ and $C_2(t)$ equals $\frac{1}{2}d(\widehat{C}_1,\widehat{C}_2)$.
\end{lemma}
\begin{proof}
The proof follows from the fact that the Hopf fibration is a Riemannian submersion between $\mathbb{S}^3$ and $\mathbb{S}^2_{\frac{1}{2}}=\{(x,y,z)\in\mathbb{R}^3 | x^2+y^2+z^2=\frac{1}{4}\}$ with their standard Riemannian metrics of sectional curvature $+1$ and $+4$ respectively, see Proposition~1.1 and Proposition~1.2 of \cite{GluckZiller}.
\end{proof}

Every rotation about a fibre of the Hopf fibration induces a rotation about a point of its base.
\begin{lemma}\label{lemma:pullbackrotation}
Given a rotation $\langle A, B \rangle \in SU_2(\mathbb{C})\times SU_2(\mathbb{C})$ about a fibre $C(t)$ of the Hopf fibration, the transformation $A \in SU_2(\mathbb{C})$ induces a rotation of $\mathbb{S}^2$ about the point to which $C(t)$ projects under the Hopf map.
\end{lemma}
\begin{proof}
Rotation about the fibre $C(t) = M(\psi, \theta) F(t)$ which projects to the point $(\psi, \theta) \in \mathbb{S}^2$ has the form
\begin{equation*}
\langle A, B \rangle = \langle \overline{M(\psi, \theta)}R(\omega)M(\psi, \theta)^t, R(\omega)\rangle.
\end{equation*}
Observe that the rotation $\langle R(\omega), R(\omega) \rangle$ fixes the geodesic $F(t)$ in $\mathbb{S}^3$ and $R(\omega)$ fixes the point $\widehat{F} = \left(\begin{array}{cc}0 & i \\i & 0 \\\end{array}\right)$ in $\mathbb{S}^2$.
Thus $A \in SU_2(\mathbb{C})$ fixes the point $\widehat{C} = M(\psi,\theta)\widehat{F}\overline{M(\psi,\theta)^t}.$ By a straightforward computation, we obtain that
\begin{equation*}
\widehat{C} = \left(\begin{array}{cc}i \cos\psi \sin\theta & \sin\theta \sin\psi + i \cos\theta \\-\sin\theta \sin\psi + i \cos\theta & -i \cos\psi \sin\theta \\\end{array}\right).
\end{equation*}
The point $\widehat{C} \in \mathbb{S}^2$ corresponds to $(\psi, \theta)$ w.r.t. the polar coordinates.
\end{proof}
\section{Examples of rigidity and flexibility}

In this section we work out two principal examples of Seifert fibred cone-manifolds: the first represents a rigid cone-manifold, the second one is flexible.
\subsection{Case of rigidity: the cone-manifold $\mathcal{H}_3(\alpha,\beta,\gamma)$}

Let $\mathcal{H}_3(\alpha,\beta,\gamma)$ denote a three-dimensional cone-manifold with underlying space the sphere $\mathbb{S}^3$ and singular locus formed by the link $\mathcal{H}_3$ with cone angles $\alpha$, $\beta$ and $\gamma$ along its components.
The remaining discussion is devoted to the proof of
\begin{theorem}\label{theorem:rigidcase}
The cone-manifold $\mathcal{H}_3(\alpha,\beta,\gamma)$ admits a spherical structure if the following inequalities are satisfied:
\begin{equation*}
2\pi - \gamma < \alpha + \beta < 2\pi + \gamma,
\end{equation*}
\begin{equation*}
-2\pi + \gamma < \alpha - \beta < 2\pi - \gamma.
\end{equation*}
The spherical structure on $\mathcal{H}_3(\alpha,\beta,\gamma)$ is unique (i.e. $\mathcal{H}_3(\alpha,\beta,\gamma)$ is globally rigid).

The lengths $\ell_\alpha$, $\ell_\beta$, $\ell_\gamma$ of its singular strata are pairwise equal and the following formula holds:
\begin{equation*}
\ell_\alpha = \ell_\beta = \ell_\gamma = \frac{\alpha+\beta+\gamma}{2} - \pi.
\end{equation*}

The volume of $\mathcal{H}_3(\alpha,\beta,\gamma)$ equals
\begin{equation*}
{\rm Vol}\,\mathcal{H}_3(\alpha,\beta,\gamma) = \frac{1}{2}\left( \frac{\alpha+\beta+\gamma}{2} - \pi \right)^2.
\end{equation*}
\end{theorem}
\begin{proof}
First, we construct a holonomy map for $\mathcal{H}_3(\alpha, \beta, \gamma)$. By applying Wirtinger's algorithm, one obtains the following fundamental group presentation for the link $\mathcal{H}_3$:
\begin{equation*}
\Gamma = \pi_1(\mathbb{S}^3 \setminus \mathcal{H}_3) = \langle a,b,c,h | acb = bac = cba = h, h\in Z(\Gamma) \rangle,
\end{equation*}
that is a central extension by $h$ of the thrice-punctured sphere group
\begin{equation*}
\Gamma_0 = \pi_1(\mathbb{S}^2 \setminus \{\mbox{3 points}\}) = \langle a,b,c | acb = bac = cba = \mathrm{id} \rangle.
\end{equation*}
Consider a holonomy map
\begin{equation*}
\rho : \Gamma \longmapsto {\rm Isom}^+\,\mathbb{S}^3 \cong SO_4(\mathbb{R}).
\end{equation*}

\begin{figure}[ht]
\begin{center}
\includegraphics* [totalheight=5cm]{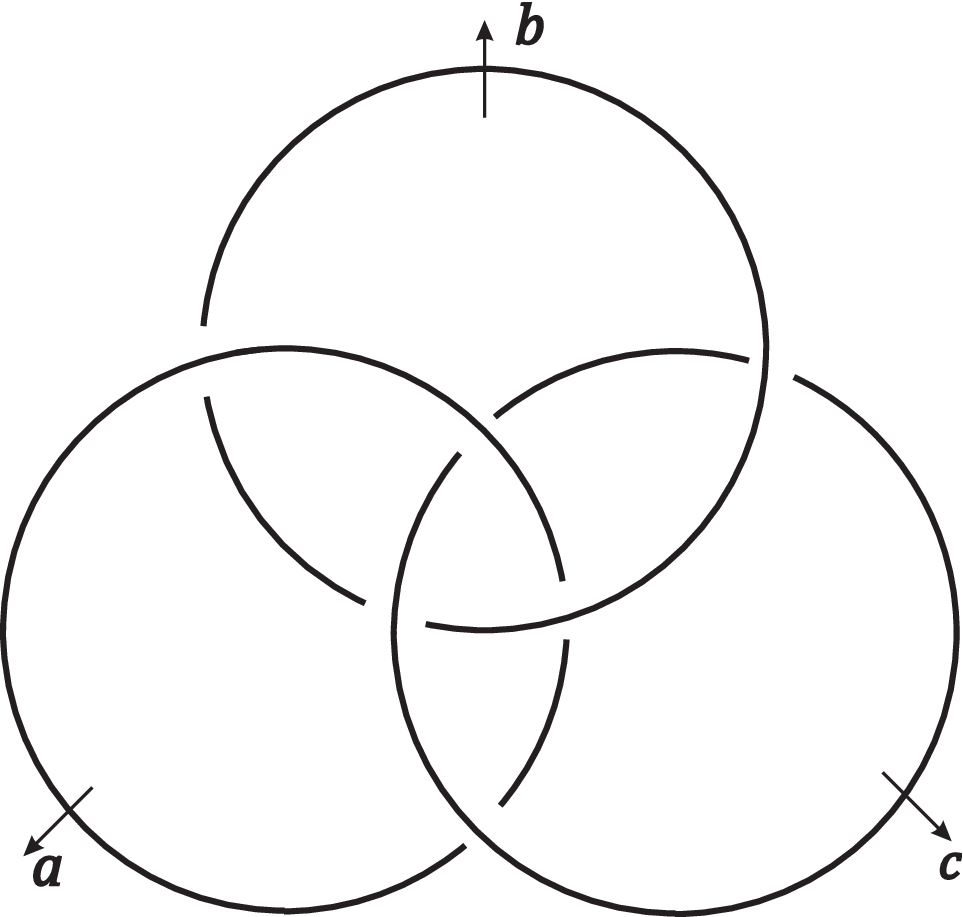}
\end{center}
\caption{The link $\mathcal{H}_3$} \label{figlink33}
\end{figure}

Let $\widetilde{\rho}$ denote its lift to $SU_2(\mathbb{C})\times SU_2(\mathbb{C})$, which is a two-fold covering of $SO_4(\mathbb{R})$ (see~\cite{Culler}):
\begin{equation*}
\widetilde{\rho} = \langle \widetilde{\rho}_1, \widetilde{\rho}_2 \rangle : \Gamma \longmapsto SU_2(\mathbb{C})\times SU_2(\mathbb{C}).
\end{equation*}

Let us note, that if holonomy images of any two generators of $\Gamma$ commute, then the whole homomorphic image $\widetilde{\rho}(\Gamma)$ is abelian. Thus, for a representation $\widetilde{\rho}$ we have that the following three cases, up to a suitable conjugation, are possible:
\begin{itemize}
\item[(i)] $\widetilde{\rho} = (\widetilde{\rho}_1, \widetilde{\rho}_2): \Gamma \rightarrow SU_2(\mathbb{C})\times SU_2(\mathbb{C})$, both $\widetilde{\rho}_1$ and $\widetilde{\rho}_2$ are non-abelian,
\item[(ii)] $\widetilde{\rho}: \Gamma \rightarrow \mathbb{S}^1\times \mathbb{S}^1$, an abelian representation,
\item[(iii)] $\widetilde{\rho} = (\widetilde{\rho}_1, \widetilde{\rho}_2): \Gamma \rightarrow SU_2(\mathbb{C})\times \mathbb{S}^1$, where $\widetilde{\rho}_1$ is non-abelian.
\end{itemize}

For case (i), let us first suppose that $\tilde{\rho}(h)$ is non-trivial. Since the holonomy  images of the meridians $a$, $b$ and $c$ have to commute with the holonomy image of $h$, they are simultaneously diagonalisable. We arrive at case (ii).

If $\widetilde{\rho}(h)$ is trivial, then we have two non-abelian representations $\widetilde{\rho}_i: \Gamma_0 \rightarrow SU_2(\mathbb{C})$. Since the holonomy images of the meridians correspond to rotations along geodesic lines in $\mathbb{S}^3$, it follows by \cite[Lemma 9.2]{BoileauLeebPorti2005} that $\mathrm{tr} \widetilde{\rho}_1(x) = \mathrm{tr} \widetilde{\rho}_2(x)$ for $x\in\{a,b,c\}$. The base space of the fibred cone-manifold $\mathcal{H}_3(\alpha,\beta,\gamma)$ is a turnover $\mathbb{S}^2(\alpha,\beta,\gamma)$, with $\alpha$, $\beta$, $\gamma$ cone angles. Then, by \cite[Lemma~4.1]{Goldman}, up to a conjugation, $\widetilde{\rho} = (\widetilde{\rho}_1, \widetilde{\rho}_1)$. The representation $\rho: \Gamma \rightarrow SO(4)$ is conjugate into $SO(3)$ and the holonomy images of the meridians have a common fixed point in $\mathbb{S}^3$. Thus, their axis intersect, which does not correspond to a non-degenerate spherical structure on the cone-manifold $\mathcal{H}_3(\alpha,\beta,\gamma)$.

For case (ii), up to a suitable conjugation, the representation $\widetilde{\rho}$ preserves the Hopf fibration. Thus, by Lemma~\ref{lemma:pullbackrotation}, it descends to an abelian representation of $\Gamma_0$, which cannot be a holonomy of a non-degenerate spherical structure on the base of the fibration. 

Finally, case (iii) is left. By \cite[Lemma 9.2]{BoileauLeebPorti2005}, one has
\begin{equation*}
\widetilde{\rho}(a) = \langle m^t_a R(\alpha) \overline{m_a}, R(\alpha) \rangle,
\end{equation*}
\begin{equation*}
\widetilde{\rho}(b) = \langle m^t_b R(\beta) \overline{m_b}, R(\beta) \rangle,
\end{equation*}
\begin{equation*}
\widetilde{\rho}(c) = \langle m^t_c R(\gamma) \overline{m_c}, R(\gamma) \rangle
\end{equation*}
for $m_a$, $m_b$, $m_c \in SU_2(\mathbb{C})$. 

Note, that every matrix $m \in SU_2(\mathbb{C})$ is of the form $m = R(\tau) M(\psi,\theta)$ for suitable $0 \leq \psi \leq \pi$, $0 \leq \theta, \tau \leq 2\pi$. Then we obtain that the image of every meridian in $\Gamma = \pi_1(\mathbb{S}^3\setminus \mathcal{H}_3)$ has the form 
\begin{equation*}
\langle m^t R(\omega) \overline{m}, R(\omega) \rangle = \langle M^t(\psi, \theta) R^t(\tau) R(\omega) \overline{R(\tau)}\, \overline{M(\psi, \tau)}, R(\omega)\rangle =
\end{equation*}
\begin{equation*}
\langle M^t(\psi, \theta) R(\omega) \overline{M(\psi, \theta)}, R(\omega) \rangle,
\end{equation*}
since $R(\omega)$ and $R(\tau)$ commute. Hence, Lemma~\ref{lemma:rotation} implies that every meridian is mapped by $\widetilde{\rho}$ to a rotation about an appropriate fibre of the Hopf fibration. By Propositions~2.1 and~2.2 of~\cite{GluckZiller}, the holonomy preserves the fibration structure.  

Let $A = \widetilde{\rho}(a)$, $B = \widetilde{\rho}(b)$, $C = \widetilde{\rho}(c)$ be holonomy images of the generators $a$, $b$, $c$ for $\Gamma = \pi_1(\mathbb{S}^3 \setminus \mathcal{H}_3)$.

After a suitable conjugation in $SU_2(\mathbb{C}) \times SU_2(\mathbb{C})$, we obtain
\begin{equation*}
A = \langle A_l, A_r \rangle = \left\langle R(\alpha), R(\alpha) \right\rangle,
\end{equation*}
\begin{equation*}
B = \langle B_l, B_r \rangle = \left\langle \overline{M(0,\phi)}R(\beta)M(0,\phi)^t, R(\beta) \right\rangle,
\end{equation*}
\begin{equation*}
C = \langle C_l, C_r \rangle = \left\langle \overline{M(\psi,\theta)}R(\gamma)M(\psi,\theta)^t, R(\gamma) \right\rangle.
\end{equation*}

In order for the holonomy map $\widetilde{\rho}$ to be a homomorphism, the following relations should hold:
\begin{equation*}
A_lC_lB_l = B_lA_lC_l = C_lB_lA_l,
\end{equation*}
\begin{equation*}
A_rC_rB_r = B_rA_rC_r = C_rB_rA_r.
\end{equation*}

The latter of them are satisfied by the construction of $\widetilde{\rho}: \Gamma \rightarrow SU_2(\mathbb{C}) \times \mathbb{S}^1$.

Let us consider the former relations. By Lemma \ref{lemma:pullbackrotation}, the elements $A_l$, $B_l$ and $C_l$ are rotations of $\mathbb{S}^2$ about the points $\widehat{F}_a = (0, 0)$, $\widehat{F}_b = (0, \phi)$ and $\widehat{F}_c = (\psi, \theta)$, respectively. Since $\widehat{F}_a$, $\widehat{F}_b$, $\widehat{F}_c$ form a triangle on $\mathbb{S}^2$ and the base space of $\mathcal{H}_3(\alpha, \beta, \gamma)$ is a turnover with $\alpha$, $\beta$, $\gamma$ cone angles, one may expect the following
\begin{lemma}\label{lemma:triangle}
The points $\widehat{F}_a = (0, 0)$, $\widehat{F}_b = (0, \phi)$ and $\widehat{F}_c = (\psi, \theta)$ form a triangle with angles $\frac{\alpha}{2}$, $\frac{\beta}{2}$ and $\frac{\gamma}{2}$ at the corresponding vertices.
\end{lemma}
\begin{proof}
By a straightforward computation, we obtain that
\begin{equation*}
A_lC_lB_l - B_lA_lC_l = \left(\begin{array}{cc}i R_1 & R_2 + i R_3 \\-R_2 + i R_3 & -i R_1 \\\end{array}\right),
\end{equation*}
\begin{equation*}
C_lB_lA_l - B_lA_lC_l = \left(\begin{array}{cc}i R_4 & R_5 + i R_3 \\-R_5 + i R_3 & - i R_4 \\\end{array}\right),
\end{equation*}
where
\begin{eqnarray*}
&R_1 = 2 \sin\frac{\beta}{2} \sin\frac{\gamma}{2} \sin\theta \cos\phi \sin\left( \frac{\alpha}{2} - \psi \right),
\end{eqnarray*}

\begin{eqnarray*}
&R_2  =  2 \sin\frac{\beta}{2}\left( \cos\frac{\gamma}{2}\sin\frac{\alpha}{2}\sin\phi +
\sin\frac{\gamma}{2}\left( -\cos\phi \cos\left( \frac{\alpha}{2} - \psi \right)\sin\theta +\right.\right.\\&
\left.\left.\cos\frac{\alpha}{2}\cos\theta\sin\phi \right) \right),&
\end{eqnarray*}

\begin{eqnarray*}
&R_3 = -2\sin\frac{\beta}{2} \sin\frac{\gamma}{2} \sin\theta \sin\phi \sin\left( \frac{\alpha}{2} - \psi \right),
\end{eqnarray*}

\begin{eqnarray*}
&R_4 = 2\sin\frac{\gamma}{2}\left( \cos\theta\sin\frac{\alpha}{2}\sin\frac{\beta}{2}\sin\phi - \left( \cos\frac{\beta}{2}\sin\frac{\alpha}{2} +\right.\right.\\&
\left.\left.\cos\frac{\alpha}{2}\sin\frac{\beta}{2}\cos\phi \right)\sin\theta \sin\psi \right),&
\end{eqnarray*}

\begin{eqnarray*}
&R_5 = 2\sin\frac{\gamma}{2}\left( \cos\frac{\beta}{2}\cos\psi\sin\frac{\alpha}{2}\sin\theta + \right.\\&
\left.\cos\frac{\alpha}{2}\sin\frac{\beta}{2}(\cos\phi \cos\psi \sin\theta - \cos\theta \sin\phi) \right).&
\end{eqnarray*}

In order to determine the parameters $\phi$, $\psi$ and $\theta$, one can proceed as follows: these are determined by the system of equations $R_k = 0$, $k \in \{1,\ldots,5\}$ under the restrictions $0 < \alpha,\beta,\gamma < 2\pi$ and $0 < \psi \leq 2\pi$, $0 < \theta \leq \pi$. Thus, the common solutions to $R_1$ and $R_3$ are $\psi = \frac{\alpha}{2}$ and $\psi = \frac{\alpha}{2} \pm \pi$. We claim that the cone angles in the base space of $\mathcal{H}_3(\alpha, \beta, \gamma)$ and along its fibres are the same, and choose $\psi = \frac{\alpha}{2}$.

Taking into account that $0 < \alpha,\beta,\gamma < 2\pi$ (this implies that the sine functions of half cone angles are non-zero), turn the set of relations $R_k$, $k\in\{1,\ldots,5\}$ into a new one:
\begin{eqnarray*}
&\widetilde{R}_1 = -\cos\phi \sin\frac{\gamma}{2}\sin\theta + \left( \sin\frac{\alpha}{2}\cos\frac{\gamma}{2} + \cos\frac{\alpha}{2}\sin\frac{\gamma}{2}\cos\theta \right) \sin\phi,
\end{eqnarray*}
\begin{eqnarray*}
&\widetilde{R}_2 = -\cos\theta \sin\frac{\beta}{2}\sin\phi + \left( \sin\frac{\alpha}{2}\cos\frac{\beta}{2} + \cos\frac{\alpha}{2}\sin\frac{\beta}{2}\cos\phi \right) \sin\theta.
\end{eqnarray*}

Note, that the conditions of Theorem~\ref{theorem:rigidcase} concerning cone angles are exactly the existence conditions for a spherical triangle with angles $\frac{\alpha}{2}$, $\frac{\beta}{2}$ and $\frac{\gamma}{2}$. For the latter, the following trigonometric identities (spherical cosine and sine rules) are satisfied \cite[Theorems 2.5.2 and 2.5.4]{Ratcliffe}:
\begin{equation*}
\cos\phi = \frac{\cos\frac{\gamma}{2} + \cos\frac{\alpha}{2}\cos\frac{\beta}{2}}{\sin\frac{\alpha}{2}\sin\frac{\beta}{2}},
\end{equation*}
\begin{equation*}
\cos\theta = \frac{\cos\frac{\beta}{2} + \cos\frac{\alpha}{2}\cos\frac{\gamma}{2}}{\sin\frac{\alpha}{2}\sin\frac{\gamma}{2}},
\end{equation*}
\begin{equation*}
\frac{\sin\phi}{\sin\frac{\gamma}{2}} = \frac{\sin\theta}{\sin\frac{\beta}{2}}.
\end{equation*}
These identities state that the points $\widehat{F}_a$, $\widehat{F}_b$ and $\widehat{F}_c$  form a triangle on $\mathbb{S}^2$ with angles $\frac{\alpha}{2}$, $\frac{\beta}{2}$ and $\frac{\gamma}{2}$ at the corresponding vertices. Its double provides the base turnover with cone angles $\alpha$, $\beta$ and $\gamma$ for the fibred cone-manifold $\mathcal{H}_3(\alpha,\beta,\gamma)$.

On substituting the expressions for $\cos\phi$ and $\cos\psi$ above in the relations $\widetilde{R}_k$, $k \in \{1, 2\}$ and taking into account the sine rule, one obtains that $\widetilde{R}_k =0$, $k\in\{1, 2\}$. The lemma is proven.
\end{proof}

Let $\mathcal{S}$ denote the domain of cone angles indicated in the statement of the theorem:
\begin{equation*}
\mathcal{S} = \left\{\overrightarrow{\alpha} = (\alpha, \beta, \gamma) \left| 
\begin{array}{c}
\,\,\, 2\pi - \gamma < \alpha + \beta < 2\pi + \gamma \\
-2\pi + \gamma < \alpha-\beta < 2\pi - \gamma
\end{array}\right. \right\}.
\end{equation*}
Let $\mathcal{S}^{\ast}$ denote the subset of $\mathcal{S}$, such that for every triple of cone angles $\overrightarrow{\alpha} = (\alpha, \beta, \gamma)\in\mathcal{S}^{\ast}$ there exists a spherical structure on $\mathcal{H}_3(\overrightarrow{\alpha})$. Our next step is to show that $\mathcal{S}^{\ast}$ coincides with $\mathcal{S}$.

The set $\mathcal{S}^{\ast}$ is non-empty. From \cite{Dunbar}, it follows that $\mathcal{H}_3(\pi,\pi,\pi)$ has a spherical structure. The orbifold $\mathcal{H}_3(\pi,\pi,\pi)$ is Seifert fibred and its base is a turnover with cone angles equal to $\pi$. Thus, the point $(\pi,\pi,\pi)\in\mathcal{S}$ belongs to $\mathcal{S}^{\ast}$.

The set $\mathcal{S}^{\ast}$ is open, because a deformation of the holonomy induces a deformation of the structure \cite{Porti1998}.

In order to prove that the set $\mathcal{S}^{\ast}$ is closed, we consider a sequence $\overrightarrow{\alpha}_n = (\alpha_n, \beta_n, \gamma_n)$ in $\mathcal{S}^{\ast}$ converging to $\overrightarrow{\alpha}_{\infty} = (\alpha_\infty, \beta_\infty, \gamma_\infty)$ in $\mathcal{S}$. Since every spherical cone-manifold with cone angles $\leq 2\pi$ is an Alexandrov space with curvature $\geq 1$ \cite{BurGroPer}, we obtain that the diameter of $\mathcal{H}_3(\overrightarrow{\alpha}_n)$ is bounded above: ${\rm diam}\,\mathcal{H}_3(\overrightarrow{\alpha}_n) \leq \pi$.

Let ${\rm dist}\, \mathcal{H}_3(\overrightarrow{\alpha}_n)$ denote the minimum of the mutual distances between the axis of rotations $A$, $B$ and $C$. Since $\overrightarrow{\alpha}_{\infty}\in\mathcal{S}$, we have by Lemma~\ref{lemma:triangle} that the turnover $\mathbb{S}^2(\overrightarrow{\alpha}_{\infty})$ is non-degenerate. By making use of Lemma \ref{lemma:equidistant}, one obtains that (restricting to a subsequence, if needed) for every $\overrightarrow{\alpha}_n\in\mathcal{S}$, $n=1,2,\dots$ the function ${\rm dist}\, \mathcal{H}_3(\overrightarrow{\alpha}_n)$ is uniformly bounded below away from zero:
\begin{equation*}
{\rm dist}\,\, \mathcal{H}_3(\overrightarrow{\alpha}_n) \geq d_0 > 0,\,\,\,n=1,2,\dots
\end{equation*}

Then we use the following facts \cite{BurGroPer}:
\begin{enumerate}
\item The Gromov-Hausdorff limit of Alexandrov spaces with curvature $\geq 1$, dimension $= 3$ and bounded diameter is an Alexandrov space with curvature $\geq 1$ and dimension $\leq 3$,

\item Dimension of an Alexandrov space with curvature $\geq 1$ holds the same at every point (the word ``dimension'' means Hausdorff or topological dimension, which are equal in the case of curvature $\geq 1$).
\end{enumerate}

Since ${\rm dist}\,\, \mathcal{H}_3(\overrightarrow{\alpha}_n) \geq d_0 > 0$, the sequence $\mathcal{H}_3(\overrightarrow{\alpha}_n)$ does not collapse. Thus, the cone-manifold $\mathcal{H}_3(\overrightarrow{\alpha}_{\infty})$ has a non-degenerate spherical structure and $\overrightarrow{\alpha}_{\infty} \in \mathcal{S}^{\ast}$.

The subset $\mathcal{S}^{\ast} \subset \mathcal{S}$ is non-empty, as well as both closed and open. This implies $\mathcal{S}^{\ast} = \mathcal{S}$.

Finally, we claim the following fact concerning the geometric characteristics of $\mathcal{H}_3(\alpha, \beta, \gamma)$ cone-manifold:
\begin{lemma}\label{lemma:geometryofH_3}
Let $\ell_\alpha$, $\ell_\beta$, $\ell_\gamma$ denote the lengths of the singular strata for $\mathcal{H}_3(\alpha, \beta, \gamma)$ cone-manifold with cone angles $\alpha$, $\beta$ and $\gamma$. Then
\begin{equation*}
\ell_\alpha = \ell_\beta = \ell_\gamma = \frac{\alpha + \beta + \gamma}{2} - \pi.
\end{equation*}
The volume of $\mathcal{H}_3(\alpha, \beta, \gamma)$ is
\begin{equation*}
{\rm Vol}\,\mathcal{H}_3(\alpha, \beta, \gamma) = \frac{1}{2}\left( \frac{\alpha + \beta + \gamma}{2} - \pi \right)^2.
\end{equation*}
\end{lemma}
\begin{proof}
Let us calculate the geometric parameters explicitly, using the holonomy map defined above.
First, we introduce two notions suitable for the further discussion. Given an element $M = \langle M_l, M_r \rangle \in SU_2(\mathbb{C})\times SU_2(\mathbb{C})$, one may assume that the pair of matrices $\langle M_l, M_r \rangle$ is conjugated, by means of a certain element $\langle C_l, C_r \rangle \in SU_2(\mathbb{C})\times SU_2(\mathbb{C})$, to the pair of diagonal matrices
\begin{equation*}
\left\langle \left(\begin{array}{cc}e^{i \gamma} & 0 \\0 & e^{-i \gamma} \\\end{array}\right), \left(\begin{array}{cc}e^{i \varphi} & 0 \\0 & e^{-i \varphi} \\\end{array}\right)\right\rangle
\end{equation*}
with $0 \leq \gamma, \varphi \leq \pi$.

Then call \textit{the translation length} of $M$ the quantity $\delta(M) := \varphi - \gamma$ and call \textit{the ``jump''} of $M$ the quantity $\nu(M) := \varphi + \gamma$, see \cite{HildenLosanoMontesinos} and \cite[Ch.6.4.2]{Weiss2005}. We suppose that $\varphi > \gamma$, otherwise changing $\gamma$, $\varphi$ for $2\pi - \gamma$ and $\pi - \varphi$ makes the considered tuple to have the desired form.

Recall that the representation of $\Gamma = \pi_1(\mathbb{S}^3\setminus \mathcal{H}_3)$ is
\begin{equation*}
\Gamma = \langle a, b, c, h | acb = bac = cba = h, h \in Z(\Gamma) \rangle,
\end{equation*}
where $a$, $b$, $c$ are meridians and $h$ is a longitudinal loop that represents a fibre. Denote by $H$ the image of $h$ under the holonomy map $\widetilde{\rho}$. Then we obtain
\begin{equation*}
\ell_\alpha = \ell_\beta = \ell_\gamma = \delta(H).
\end{equation*}

Since $A = \widetilde{\rho}(a)$ and $H = \widetilde{\rho}(h)$ commute, there exists an element $C = \langle C_l, C_r\rangle$ of $SU_2(\mathbb{C}) \times SU_2(\mathbb{C})$ such that
\begin{equation*}
CAC^{-1} = \left\langle \left(\begin{array}{cc}e^{i \frac{\alpha}{2}} & 0 \\0 & e^{-i \frac{\alpha}{2}} \\\end{array}\right),
\left(\begin{array}{cc}e^{i \frac{\alpha}{2}} & 0 \\0 & e^{-i \frac{\alpha}{2}} \\\end{array}\right) \right\rangle,
\end{equation*}
\begin{equation*}
CHC^{-1} = \left\langle \left(\begin{array}{cc}e^{i \gamma(H)} & 0 \\0 & e^{-i \gamma(H)} \\\end{array}\right), \left(\begin{array}{cc}e^{i \varphi(H)} & 0 \\0 & e^{-i \varphi(H)} \\\end{array}\right)\right\rangle.
\end{equation*}
By a straightforward computation similar to that in Lemma \ref{lemma:triangle}, one obtains
\begin{equation*}
2 \cos \gamma(H) = \mathrm{tr} H_l = \mathrm{tr} A_l C_l B_l = \mathrm{tr}( -\mathrm{id}) = 2 \cos \pi
\end{equation*}
and
\begin{equation*}
2 \cos \varphi(H) = \mathrm{tr} H_r = \mathrm{tr} A_r C_r B_r = 2\cos\frac{\alpha + \beta + \gamma}{2}.
\end{equation*}
From the foregoing discussion, the singular stratum's length is
\begin{equation*}
\ell_\alpha = \delta(H) = \frac{\alpha + \beta + \gamma}{2} - \pi.
\end{equation*}
An analogous equality holds for $\ell_\beta$ and $\ell_\gamma$.

By the Schl\"{a}fli formula \cite{Hodgson}, the following relation holds:
\begin{equation*}
2\,\,{\rm d Vol}\,\mathcal{H}_3(\alpha, \beta, \gamma) = \ell_\alpha {\rm d}\alpha + \ell_\beta {\rm d}\beta + \ell_\gamma {\rm d}\gamma.
\end{equation*}
Solving this differential equality, we obtain that
\begin{equation*}
{\rm Vol}\,\mathcal{H}_3(\alpha, \beta, \gamma) = \frac{1}{2}\left( \frac{\alpha + \beta + \gamma}{2} - \pi \right)^2 + {\rm Vol_0},
\end{equation*}
where ${\rm Vol_0}$ is an arbitrary constant.
Since the geometric structure on the base space of the fibration (consequently, on the whole $\mathcal{H}_3(\alpha, \beta, \gamma)$ cone-manifold) degenerates when $\alpha + \beta + \gamma \longrightarrow 2\pi$, the equality ${\rm Vol_0} = 0$ follows from the volume function continuity.
\end{proof}

Consider a holonomy $\widetilde{\rho} = \langle \widetilde{\rho}_1, \widetilde{\rho}_2 \rangle: \Gamma = \pi_1(\mathbb{S}^3\setminus \mathcal{H}_3) \rightarrow SU_2(\mathbb{C})\times SU_2(\mathbb{C})$ for $\mathcal{H}_3(\alpha, \beta, \gamma)$ cone-manifold. As we already know from the preceding discussion, one has $\widetilde{\rho}: \Gamma \rightarrow SU_2(\mathbb{C})\times \mathbb{S}^1$ essentially, and $\widetilde{\rho}_1$ determines $\widetilde{\rho}_2$ up to a conjugation by means of the equality $\tr \widetilde{\rho}_1(m) = \tr \widetilde{\rho}_2(m)$ for meridians in $\Gamma$. So any deformation of $\widetilde{\rho}$ is a deformation of $\widetilde{\rho}_1$. In the case of $\mathcal{H}_3(\alpha, \beta, \gamma)$, the map $\widetilde{\rho}_1$ is a non-abelian representation of the base turnover group. Spherical turnover is rigid, that means $\widetilde{\rho}_1$ is determined only by the corresponding cone angles. Thus $\mathcal{H}_3(\alpha, \beta, \gamma)$ is locally rigid. 

The global rigidity follows from the fact that every $\mathcal{H}_3(\alpha, \beta, \gamma)$ cone-manifold could be deformed to the orbifold $\mathcal{H}_3(\pi, \pi, \pi)$ by a continuous path through locally rigid structures. This assertion holds since $\mathcal{S}^{\ast}$ contains the point $(\pi, \pi, \pi)$ and $\mathcal{S}^{\ast}$ is convex. The global rigidity of $\mathcal{H}_3(\pi, \pi, \pi)$ spherical orbifold follows from \cite{deRham, Rothenberg} and implies the global rigidity of $\mathcal{H}_3(\alpha, \beta, \gamma)$ by means of deforming the orbifold structure backwards to the considered cone-manifold one.
\end{proof}
\subsection{Case of flexibility: the cone-manifold $\mathcal{H}_4(\alpha)$}

Let $\mathcal{H}_4(\alpha)$ denote a three-dimensional cone-manifold with underlying space the sphere $\mathbb{S}^3$ and singular locus formed by the link $\mathcal{H}_4$ with cone angle $\alpha$ along all its components.

\begin{figure}[ht]
\begin{center}
\includegraphics* [totalheight=6cm]{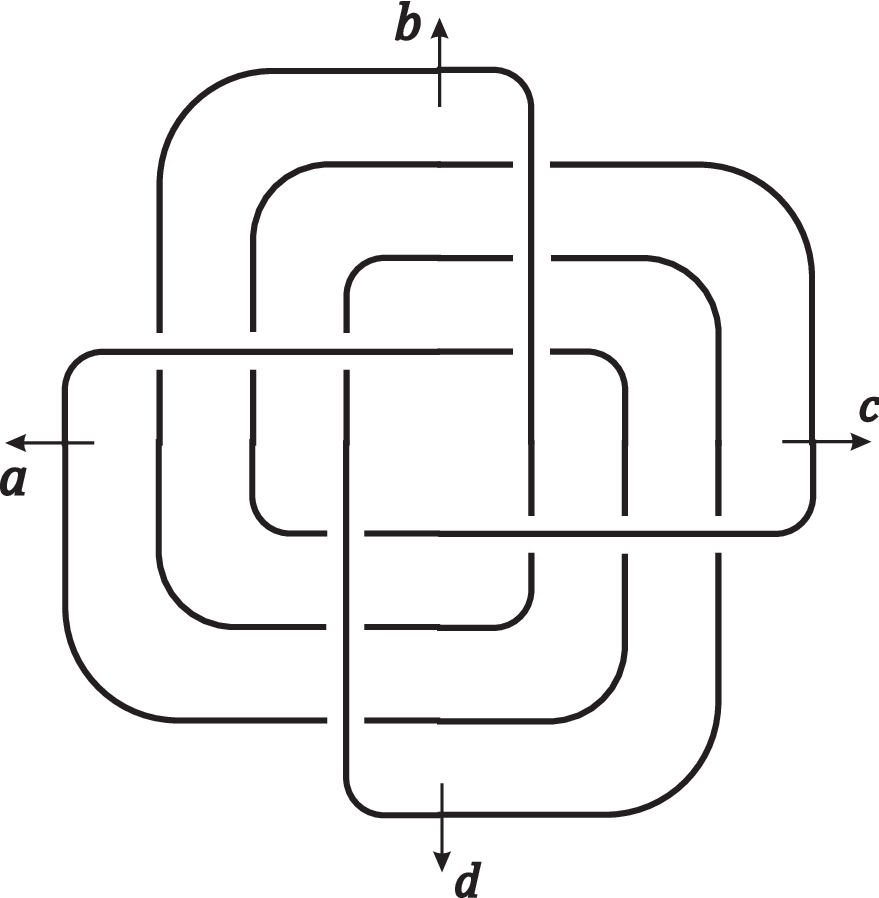}
\end{center}
\caption{The link $\mathcal{H}_4$} \label{figlink44}
\end{figure}

The following theorem provides an example of a flexible cone-manifold which is Seifert fibred.

\begin{theorem}\label{theorem:nonrigidcase}
The cone-manifold $\mathcal{H}_4(\alpha)$ admits a spherical structure if
\begin{equation*}
\pi < \alpha < 2 \pi.
\end{equation*}
This structure is not unique (i.e. $\mathcal{H}_4(\alpha)$ is not globally, nor locally rigid). The deformation space contains an open interval, that provides a one-parameter family of distinct spherical cone-metrics on $\mathbb{S}^3$.

The length of each singular stratum is
\begin{equation*}
\ell = 2(\alpha - \pi).
\end{equation*}

The volume of $\mathcal{H}_4(\alpha)$ equals
\begin{equation*}
{\rm Vol}\,\mathcal{H}_4(\alpha) = 2(\alpha - \pi)^2.
\end{equation*}
\end{theorem}
\begin{proof}
The following lemma precedes the proof of the theorem.

\begin{lemma}\label{lemma:quadrangle}
Given a quadrangle $Q$ on $\mathbb{S}^2$ with three right angles and one angle $\frac{\alpha}{2}$ (see Fig. \ref{quadr}), the following statements hold:
\begin{enumerate}
\item The quadrangle $Q$ exists if $\pi < \alpha < 2\pi$,
\item $\sin\ell_1 \sin\ell_2 = - \cos\frac{\alpha}{2}$,
\item $\cos\phi = \frac{\cos\ell_1 \cos\ell_2}{\sin\frac{\alpha}{2}}$,
\item $\cos\psi = \tan\ell_1 \cot\phi$,
\item $0\leq \ell_1,\,\ell_2,\,\phi,\,\psi \leq \frac{\pi}{2}$.
\end{enumerate}
\end{lemma}
\begin{proof}
We refer the reader to \cite[\S\,3.2]{Vinberg} for a detailed proof of the statements above.
\end{proof}

\begin{figure}[ht]
\begin{center}
\includegraphics* [totalheight=4cm]{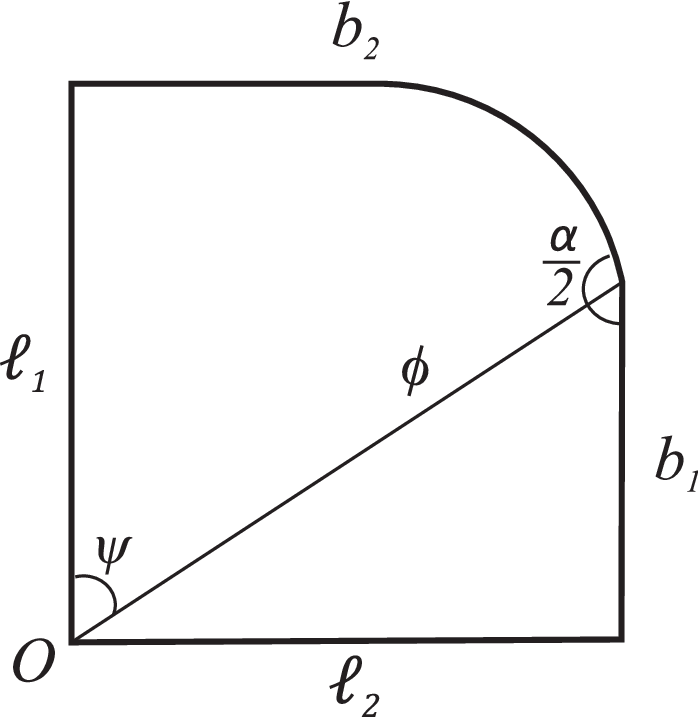}
\end{center}
\caption{The quadrangle $Q$} \label{quadr}
\end{figure}

Given a quadrangle $Q$ from Lemma \ref{lemma:quadrangle} (so-called Saccheri's quadrangle) one can construct another one, depicted in Fig.\ref{quadrlarge}, by reflecting $Q$ in its sides incident to the vertex $O$. We may regard $O$ to be the point $(0,0) \in \mathbb{S}^2$. Thus, the fibres over the corresponding vertices are
\begin{equation*}
F_a(t) = M(\psi, \phi)\,F(t),
\end{equation*}
\begin{equation*}
F_b(t) = M(\pi-\psi, \phi)\,F(t),
\end{equation*}
\begin{equation*}
F_c(t) = M(\pi+\psi, \phi)\,F(t),
\end{equation*}
\begin{equation*}
F_d(t) = M(2\pi-\psi, \phi)\,F(t).
\end{equation*}

\begin{figure}[ht]
\begin{center}
\includegraphics* [totalheight=6cm]{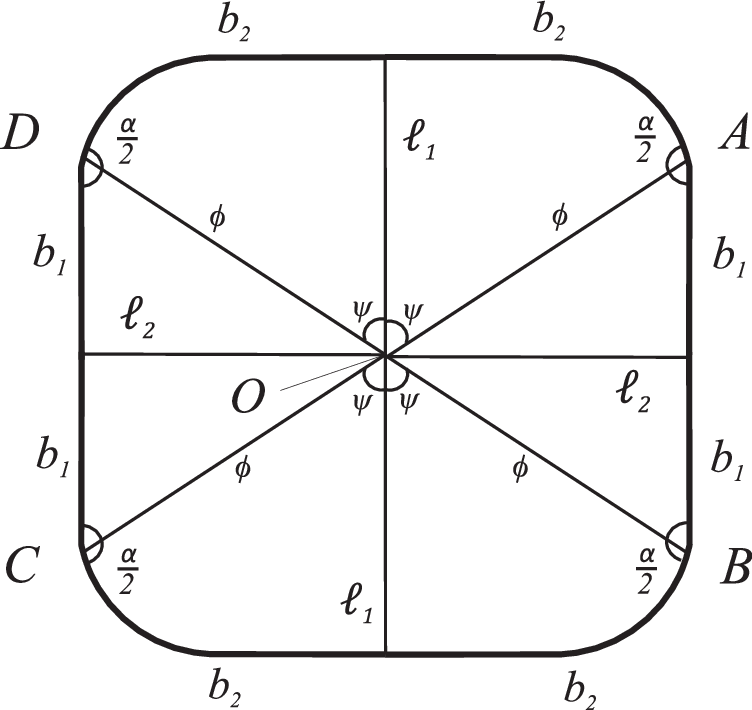}
\end{center}
\caption{The base quadrangle $P$ for $\mathcal{H}_4(\alpha)$} \label{quadrlarge}
\end{figure}

Let $A = \langle A_l, A_r \rangle$, $B = \langle B_l, B_r \rangle$, $C = \langle C_l, C_r \rangle$, $D = \langle D_l, D_r \rangle$ denote the respective rotations through angle $\alpha$ about the axis $F_a$, $F_b$, $F_c$ and $F_d$. From Lemma \ref{lemma:rotation}, one obtains
\begin{equation*}
A_l = \overline{M(\psi, \phi)}\,R(\alpha)\,M(\psi, \phi)^t,\,A_r = R(\alpha);
\end{equation*}
\begin{equation*}
B_l = \overline{M(\pi-\psi, \phi)}\,R(\alpha)\,M(\pi-\psi, \phi)^t,\,B_r = R(\alpha);
\end{equation*}
\begin{equation*}
C_l = \overline{M(\pi+\psi, \phi)}\,R(\alpha)\,M(\pi+\psi, \phi)^t,\,B_r = R(\alpha);
\end{equation*}
\begin{equation*}
D_l = \overline{M(2\pi-\psi, \phi)}\,R(\alpha)\,M(2\pi-\psi, \phi)^t,\,D_r = R(\alpha).
\end{equation*}
We assume that $\ell_1$, $\ell_2$, $\phi$ and $\psi$ satisfy the identities of Lemma \ref{lemma:quadrangle}. 

The fundamental group of $\pi_1(\mathbb{S}^3\setminus \mathcal{H}_4)$ has the presentation
\begin{equation*}
\Gamma = \pi_1(\mathbb{S}^3\setminus \mathcal{H}_4) = \langle a,b,c,d,h | adcb = badc = cbad = dcba = h, h\in Z(\Gamma) \rangle.
\end{equation*}

Let us construct a lift of the holonomy map $\widetilde{\rho}: \Gamma \rightarrow SU_2(\mathbb{C})\times SU_2(\mathbb{C})$ as follows:
\begin{equation*}
\widetilde{\rho}(a) = A,\,\widetilde{\rho}(b) = B,\,\widetilde{\rho}(c) = C,\,\widetilde{\rho}(d) = D.
\end{equation*}
Here we choose $\widetilde{\rho}: \Gamma \rightarrow SU_2(\mathbb{C})\times \mathbb{S}^1$ by the same reason as in Theorem~\ref{theorem:rigidcase}.

In order to show that the map $\widetilde{\rho}$ is a homomorphism, one has to check whether the following relations are satisfied:
\begin{equation*}
A_lD_lC_lB_l = B_lA_lD_lC_l = C_lB_lA_lD_l = D_lC_lB_lA_l,
\end{equation*}
\begin{equation*}
A_rD_rC_rB_r = B_rA_rD_rC_r = C_rB_rA_rD_r = D_rC_rB_rA_r.
\end{equation*}

The latter relations hold in view of the fact that the matrices $A_r$, $B_r$, $C_r$ and $D_r$ pairwise commute. Then, we show that the following equality holds:
\begin{equation*}
A_lD_lC_lB_l = \mathrm{id}.
\end{equation*}

\begin{figure}[t]
\begin{center}
\includegraphics* [totalheight=4.5cm]{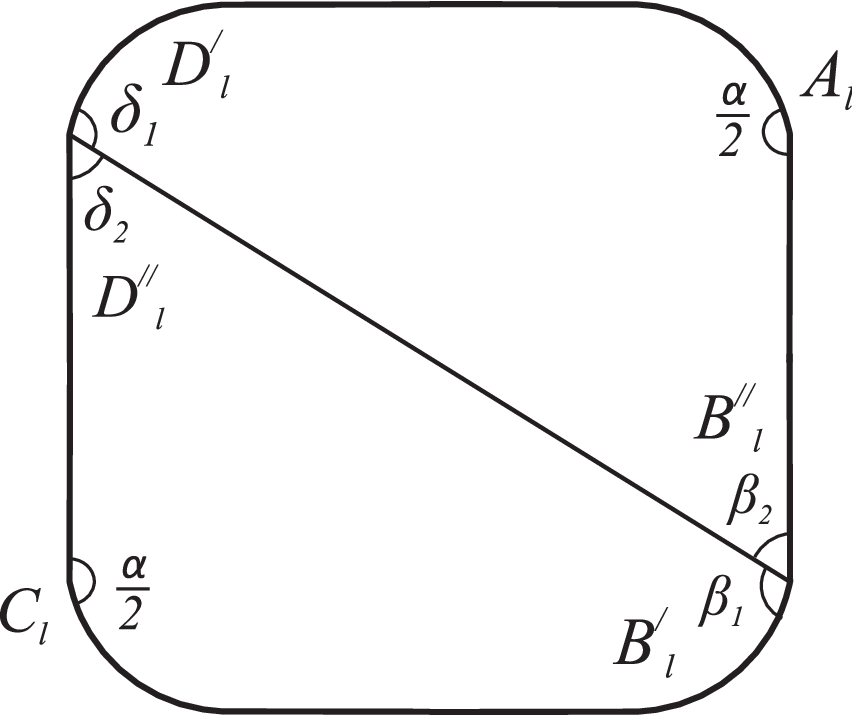}
\end{center}
\caption{Section of P by the line joining vertices $B$ and $D$} \label{quadrsection}
\end{figure}

To do this, split the quadrangle $P$ into two triangles by drawing a geodesic line from $B$ to $D$. Since $A_l$, $B_l$, $C_l$ and $D_l$ are rotations about the vertices of the quadrangle depicted in Fig. \ref{quadrsection}, let us decompose the rotations $B_l = B'_l B''_l$ and $D_l = D'_l D''_l$ into the products of rotations $B'_l$, $B''_l$ through angles $\beta_1$, $\beta_2$ and the rotations $D'_l$, $D''_l$ through angles $\delta_1$ and $\delta_2$, respectively. The following equalities hold: $\beta_1 + \beta_2 = \frac{\alpha}{2}$ and $\delta_1 + \delta_2 = \frac{\alpha}{2}$. Thus, the triples $D''_l$, $C_l$, $B'_l$ and $A_l$, $D'_l$, $B''_l$ consist of rotations about the vertices of two disjoint triangles depicted in Fig. \ref{quadrsection}. Similar to the computation of Lemma \ref{lemma:geometryofH_3}, we have
\begin{equation*}
D''_l C_l B'_l = - \mathrm{id}
\end{equation*}
and
\begin{equation*}
A_l D'_l B''_l = - \mathrm{id}.
\end{equation*}
From the identities above, it follows that
\begin{equation*}
A_l D_l C_l B_l = A_l D'_l D''_l C_l B'_l B''_l = - A_l D'_l B''_l = \mathrm{id}.
\end{equation*}
The statement holds under a cyclic permutation of the factors. Thus,
\begin{equation*}
A_lD_lC_lB_l = B_lA_lD_lC_l = C_lB_lA_lD_l = D_lC_lB_lA_l = \mathrm{id}.
\end{equation*}

Below we shall consider the side-length $\ell_1$ as a parameter. Let $\ell_1 := \tau$.  Then by Lemma \ref{lemma:quadrangle} one has that $\sin\ell_2 = -\frac{\cos\frac{\alpha}{2}}{\sin\tau}$ and $\ell_2 := \ell_2(\tau)$ is a well-defined continuous function of $\tau$. The quadrangle $P$ depends on the parameter $\tau$ continuously while keeping the angles in its vertices equal to~$\frac{\alpha}{2}$.

Let $\mathcal{H}_4(\alpha; \tau)$ denote a three-dimensional cone-manifold with underlying space the sphere $\mathbb{S}^3$ and singular locus the link $\mathcal{H}_4$ with cone angle $\alpha$ along its components. Furthermore, its holonomy map is determined by the quadrangle $P$ described above (see Fig.~\ref{quadrlarge}) depending on the parameter $\tau$. This means that the double of $P$ forms a ``pillowcase'' cone-surface with all cone angles equal to $\alpha$, which is the base space for the fibred cone-manifold $\mathcal{H}_4(\alpha; \tau)$.

Let $\mathbb{L}_n(\alpha, \beta)$ be a cone-manifold with underlying space the sphere $\mathbb{S}^3$ and singular locus a torus link of the type $(2, 2n)$ with cone angles $\alpha$ and $\beta$ along its components. Torus links of the type $(2, 2n)$ are two-bridge links. The corresponding cone-manifolds were previously considered in \cite{KolpakovMednykh, Porti2004}. Since the cone-manifold $\mathcal{H}_4(\alpha)$ forms a 4-fold branched covering of the cone-manifold $\mathbb{L}_4(\alpha, \frac{\pi}{2})$, from \cite[Theorem~2]{KolpakovMednykh} we obtain that $\mathcal{H}_4(\alpha)$ has a spherical structure if $\pi < \alpha < 2\pi$. The length of each singular stratum equals to $\ell = 2(\alpha - \pi)$ and the volume is ${\rm Vol}\,\mathcal{H}_4(\alpha) = 2 (\alpha - \pi)^2$.

Under the assumption that $\ell_1 = \ell_2$, the base quadrangle depicted in Fig.~\ref{quadrlarge} appears to have a four order symmetry. Moreover, by making use of Lemma \ref{lemma:quadrangle}, one may derive the following equalities: $\psi = \frac{\pi}{4}$, $\cos\phi = \cot\frac{\alpha}{4}$. The general formulas for the holonomy of $\mathcal{H}_4(\alpha)$ cone-manifold derived above subject to the condition $\ell_1 = \ell_2$ (equivalently, the cone-manifold $\mathcal{H}_4(\alpha)$ has a four order symmetry) give the holonomy map induced by the covering. Thus $\mathcal{H}_4(\alpha) \cong \mathcal{H}_4(\alpha; \arccos(\sqrt{2}\cos\frac{\alpha}{4}))$ is a spherical cone-manifold. 

We claim that one can vary the parameter $\tau$ in certain ranges while keeping spherical structure on $\mathcal{H}_4(\alpha; \tau)$ non-degenerate. 

\begin{lemma}\label{lemma:nondegenerateH_4}
If $\tau$ varies over $(\frac{\alpha-\pi}{2}, \frac{\pi}{2})$, the cone-manifold $\mathcal{H}_4(\alpha; \tau)$ has a non-degenerate spherical structure.
\end{lemma}
\begin{proof}
The proof has much in common with the proof of the spherical structure existence on $\mathcal{H}_3(\alpha, \beta, \gamma)$ cone-manifold given in Theorem \ref{theorem:rigidcase}. Let us express the identities of Lemma \ref{lemma:quadrangle} in terms of the parameter $\ell_1 := \tau$. We obtain
\begin{equation*}
\cos\phi = \cos\tau \sqrt{1 - \cot^2\frac{\alpha}{2} \cot^2\tau},
\end{equation*}
\begin{equation*}
\cos\psi = \sqrt{\frac{1-\cot^2\frac{\alpha}{2}\cot^2\tau}{1+\cot^2\frac{\alpha}{2}\cot^4\tau}},
\end{equation*}
\begin{equation*}
\sin\ell_2 = - \frac{\cos\frac{\alpha}{2}}{\sin\tau}.
\end{equation*}

Since Lemma \ref{lemma:quadrangle} states that $0 \leq \phi,\,\psi,\,\ell_2 \leq \frac{\pi}{2}$, the functions $\phi := \phi(\tau)$, $\psi := \psi(\tau)$, $\ell_2 := \ell_2(\tau)$ are well-defined and depend continuously on $\tau$.

Moreover, the following relations hold:
\begin{equation*}
\cos b_1 = \frac{\cos\phi}{\cos\ell_2} = \cos\tau \sqrt{\frac{\sin^2\tau - \cot^2\frac{\alpha}{2}\cos^2\tau}{\sin^2\tau - \cos^2\frac{\alpha}{2}}},
\end{equation*}
\begin{equation*}
\cos b_2 = \frac{\cos\phi}{\cos\tau} = \sqrt{1 - \cot^2\frac{\alpha}{2}\cot^2\tau}.
\end{equation*}
If one sets the centre $O$ of the quadrangle $P$ to $(0, 0) \in \mathbb{S}^2$, the whole quadrangle is situated in the upper hemisphere provided $\phi < \frac{\pi}{2}$. From the fact that $\cos b_1 \geq \cos\phi$ and $\cos b_2 \geq \cos\phi$, it follows $b_1,\,b_2 \leq \phi$. Thus $b_1,\,b_2 \leq \frac{\pi}{2}$ and the functions $b_1 := b_1(\tau)$, $b_2 := b_2(\tau)$ are well-defined and continuous with respect to $\tau$.

Observe that if the condition $\frac{\alpha - \pi}{2} < \tau < \frac{\pi}{2}$ is satisfied, then the required inequality $\phi < \frac{\pi}{2}$ holds.

Let $\mathcal{S^\ast_\alpha}$ denote the subset of $\mathcal{S}_\alpha = \{\tau | \frac{\alpha-\pi}{2} < \tau < \frac{\pi}{2}\}$ that consists of the points $\tau \in \mathcal{S}_\alpha$ such that the cone-manifold $\mathcal{H}_4(\alpha; \tau)$ has a non-degenerate spherical structure. We show $\mathcal{S}^\ast_\alpha = \mathcal{S}_\alpha$ by means of the fact that $\mathcal{S}^\ast_\alpha$ is both open and closed non-empty subset of $\mathcal{S}_\alpha$.

As noticed above, $\tau = \arccos(\sqrt{2}\cos\frac{\alpha}{4})$ belongs to $\mathcal{S}^\ast_\alpha$. Hence the set $\mathcal{S}^\ast_\alpha$ is non-empty.

The set $\mathcal{S}^\ast_\alpha$ is open by the fact that a deformation of the holonomy implies a deformation of the structure \cite{Porti1998}. To prove that $\mathcal{S}^\ast_\alpha$ is closed, consider a sequence $\tau_n$ converging in $\mathcal{S}^\ast_\alpha$ to $\tau_\infty \in \mathcal{S}_\alpha$.

The lengths of common perpendiculars between the axis of rotations $A$, $B$, $C$ and $D$ defined above equal respectively $b_1$, $b_2$ and $\phi$.

Since $\tau_{\infty}$ corresponds to a non-degenerated quadrangle, every cone-manifold $\mathcal{H}_4(\alpha; \tau_n)$ has the quantities $b_1(\tau_n)$, $b_2(\tau_n)$ and $\phi(\tau_n)$ uniformly bounded below away from zero. By the arguments similar to those of Theorem \ref{theorem:rigidcase}, we obtain that $\mathcal{H}_4(\alpha; \tau_\infty)$ is a non-degenerate spherical cone-manifold. Thus $\tau_\infty$ belongs to $\mathcal{S}^\ast_\alpha$. Hence $\mathcal{S}^\ast_\alpha$ is closed.

Finally, we obtain that $\mathcal{S}^\ast_\alpha = \mathcal{S}_\alpha$. Thus, while $\tau$ varies over $(\frac{\alpha-\pi}{2}, \frac{\pi}{2})$ the cone-manifold $\mathcal{H}_4(\alpha; \tau)$ does not collapse.
\end{proof}

The following lemma shows that the interval $(\frac{\alpha-\pi}{2}, \frac{\pi}{2})$ represents a part of the deformation space for possible spherical structures on $\mathcal{H}_4(\alpha; \tau)$.

\begin{lemma}\label{lemma:deformationspaceofH_4}
The cone-manifolds $\mathcal{H}_4(\alpha; \tau_1)$ and $\mathcal{H}_4(\alpha; \tau_2)$ with $\pi < \alpha < 2\pi$ and $\frac{\alpha-\pi}{2} < \tau_1, \tau_2 < \frac{\pi}{2}$ are not isometric if $\tau_1 \neq \tau_2$.
\end{lemma}
\begin{proof}
If the cone-manifolds $\mathcal{H}_4(\alpha; \tau_1)$ and $\mathcal{H}_4(\alpha; \tau_2)$ were isometric, then their holonomy maps $\widetilde{\rho}_i$, $i=1, 2$ would be conjugated representations of $\Gamma = \pi_1(\mathbb{S}^3\setminus \mathcal{H}_4)$ into $SU_2(\mathbb{C})\times SU_2(\mathbb{C})$. Then the mutual distances between the axis of rotations $A_i$, $B_i$, $C_i$ and $D_i$, $i=1, 2$, coming from the holonomy maps $\widetilde{\rho}_1$ and $\widetilde{\rho}_2$ would be equal for the corresponding pairs. From Lemma~\ref{lemma:equidistant}, it follows that the common perpendicular length for the given fibres $C_1$ and $C_2$ is half the distance between the images of $C_1$ and $C_2$ under the Hopf map. By applying Lemmas~\ref{lemma:equidistant} and \ref{lemma:nondegenerateH_4} to the base quadrangle $P$ of $\mathcal{H}_4(\alpha; \tau_i)$, $i=1, 2$ one makes sure that the inequality $\tau_1 \neq \tau_2$ implies the inequality for the lengths of corresponding common perpendiculars.
\end{proof}

Note, that by the Schl\"{a}fli formula the volume of $\mathcal{H}_4(\alpha)$ remains the same under any deformation preserving cone angles. Then the formulas for the volume and the singular stratum length follow from the covering properties of $\mathcal{H}_4(\alpha) \stackrel{4:1}{\rightarrow} \mathbb{L}_4(\alpha, \frac{\pi}{2})$ and Theorem~2 of \cite{KolpakovMednykh}. Thus, Theorem \ref{theorem:nonrigidcase} is proven.
\end{proof}

\flushleft{\textit{
Alexander Kolpakov \\
Department of Mathematics \\
University of Fribourg\\
chemin du Mus\'{e}e 23\\
CH-1700 Fribourg, Switzerland\\
}
\rm{kolpakov.alexander@gmail.com}


\begin{thebibliography}{}
\bibitem{BoileauLeebPorti2001}\textsc{Boileau, M., Leeb, B., Porti, J.} {``Uniformization of small 3-orbifolds''}. C.R.~Acad. Sci. Paris Se'r. I Math. \textbf{332}(1), 57-62 (2001).

\bibitem{BoileauLeebPorti2005}\textsc{Boileau, M., Leeb, B., Porti, J.} {``Geometrization of 3-dimensional orbifolds''}. Ann.~Math. \textbf{162}(1), 195-250 (2005).

\bibitem{BurGroPer}\textsc{Burago, Yu., Gromov, M., Perelman, G.} {``A.~D. Aleksandrov spaces with curvature bounded below''}. Russian Math. Surveys \textbf{47}, 1-58 (1992).

\bibitem{BurdeMurasugi}\textsc{Burde, G., Murasugi, K.} {``Links and Seifert fiber spaces''}. Duke Math.~J. \textbf{37}(1), 89-93 (1970).

\bibitem{Casson}\textsc{Casson, A.} {``An example of weak non-rigidity for cone manifolds with vertices''}. Talk at the Third MSJ regional workshop, Tokyo, 1998. 

\bibitem{CooperHodgsonKerckhoff}\textsc{Cooper, D., Hodgson, C., Kerckhoff, S.} {``Three-dimensional orbifolds and cone-manifolds'', Postface by S.~Kojima}. Tokyo: Mathematical Society of Japan, 2000. (MSJ Memoirs; 5)

\bibitem{Culler}\textsc{Culler, M.} {``Lifting representations to covering groups''}. Adv.~Math. \textbf{59}(1), 64-70 (1986).

\bibitem{Dunbar}\textsc{Dunbar, W.D.} {``Geometric orbifolds''}. Rev. Mat. Univ. Complut. Madrid \textbf{1}, 67-99 (1988).

\bibitem{GluckZiller}\textsc{Gluck, H., Ziller, W.} {``The geometry of the Hopf fibrations''}. L'Enseign. Math. \textbf{32}, 173-198 (1986).

\bibitem{Goldman}\textsc{Goldman, W.} {``Ergodic Theory on Moduli Spaces''}. Ann.~Math. \textbf{146}(3), 475-507 (1997).

\bibitem{HildenLosanoMontesinos}\textsc{Hilden, H.M., Lozano, M.T., Montesinos-Amilibia, J.-M.} {``Volumes and Chern-Simons invariants of cyclic coverings over rational knots''}. Proceedings of the 37-th Taniguchi Symposium on Topology and Teichmuller Spaces held in Finland, July 1995, ed. by Sadayoshi Kojima et al. (1996), 31-35.

\bibitem{Hodgson}\textsc{Hodgson, C.} {``Degeneration and regeneration of Hyperbolic Structures on Three-Manifolds''}. Princeton: Thesis, 1986.

\bibitem{HK98}\textsc{Hodgson, C., Kerckhoff, S.} {``Rigidity of hyperbolic cone-manifolds and hyperbolic Dehn surgery''}. J.~Differential Geom. \textbf{48}(1), 1-59 (1998).

\bibitem{Hopf}\textsc{Hopf, H.} {``\"{U}ber die Abbildungen der dreidimensionalen Sph\"{a}re auf die Kugelfl\"{a}che''}. Math.~Ann. \textbf{104}, 637-665 (1931).

\bibitem{Izmestiev2009}\textsc{Izmestiev, I.} {``Examples of infinitesimally flexible 3-dimensional hyperbolic cone-manifolds''}. J.~Math.~Soc.~Japan \textbf{63}(2), 581-598 (2011); arXiv:0910.2876.

\bibitem{Kojima}\textsc{Kojima, S.} {``Deformations of hyperbolic 3-cone-manifolds''}. J.~Diff. Geom. \textbf{49}(3), 469-516 (1998).

\bibitem{KolpakovMednykh}\textsc{Kolpakov, A.A., Mednykh, A.D.} {``Spherical structures on torus knots and links''}. Siberian Math.~J. \textbf{50}(5), 856-866 (2009).

\bibitem{M2009}\textsc{Montcouquiol, G.} {``Deformation of hyperbolic convex polyhedra and 3-cone-manifolds''}. arXiv:0903.4743.

\bibitem{Mostow}\textsc{Mostow, G.D.} {``Quasi-conformal mappings in n-space and the rigidity of hyperbolic space forms''}. Inst. Hautes Etudes Sci. Publ. Math. \textbf{34}, 53-104 (1968).

\bibitem{Porti1998}\textsc{Porti, J.} {``Regenerating hyperbolic and spherical cone structures from Euclidean ones''}. Topology \textbf{37}(2), 365-392 (1998).

\bibitem{Porti2002}\textsc{Porti, J.} {``Regenerating hyperbolic cone structures from Nil''}. Geom. Topol. \textbf{6}, 815-852 (2002).

\bibitem{Porti2004}\textsc{Porti, J.} {``Spherical cone structures on 2-bridge knots and links''}. Kobe J.~Math. \textbf{21}(1-2), 61-70 (2004).

\bibitem{Porti2010}\textsc{Porti, J.} {``Regenerating hyperbolic cone 3-manifolds from dimension 2''}. arXiv:1003.2494.

\bibitem{Prasad}\textsc{Prasad, G.} {``Strong rigidity of $\mathbf{Q}$-rank~1 lattices''}. Invent. Math. \textbf{21}, 255-286 (1973).

\bibitem{Ratcliffe}\textsc{Ratcliffe, J.} {``Foundations of hyperbolic manifolds''}. New York: Springer-Verlag, 1994. (Graduate Texts in Math.; 149).

\bibitem{deRham}\textsc{de~Rham, G.} {``Reidemeister's torsion invariant and rotations of $S^n$''}; in Differential Analysis, Bombay Colloq., Oxford Univ. Press, London, 1964, 27-36.

\bibitem{Rothenberg}\textsc{Rothenberg, M.} {``Torsion invariants and finite transformation groups''}. Proceedings of Symposia in Pure Math. \textbf{32}, 267-311 (1978).

\bibitem{Schlenker}\textsc{Schlenker, J.-M.} {``Dihedral angles of convex polyhedra''}. Discrete Comput. Geom. \textbf{23}, 409-417 (2000).

\bibitem{Thurston}\textsc{Thurston, W.P.} {``Geometry and topology of three-manifolds''}. Princeton Univ., 1979. (Princeton University Lecture Notes)

\bibitem{Vinberg}\textsc{Vinberg E.B., ed.} {``Geometry II. Spaces of Constant Curvature''}. New York: Springer-Verlag, 1993. (Encyclopaedia of Mathematical Sciences;~29)

\bibitem{Weiss2005}\textsc{Wei\ss, H.} {``Local rigidity of 3-dimensional cone-manifolds''}. J. Diff. Geom. \textbf{71}(3), 437-506 (2005).

\bibitem{Weiss2007}\textsc{Wei\ss, H.} {``Global rigidity of 3-dimensional cone-manifolds''}.  J. Diff. Geom. \textbf{76}(3), 495-523 (2007).

\bibitem{Weiss2009}\textsc{Wei\ss, H.} {``The deformation theory of hyperbolic cone-3-manifolds with cone-angles less than $2\pi$''}. arXiv:0904.4568.
\end{thebibliography}
\end{document}